\documentclass[oneside,english]{amsart}
\usepackage[T1]{fontenc}
\usepackage[latin9]{inputenc}
\usepackage{amsthm}
\usepackage{amssymb}
\usepackage{wasysym}
\PassOptionsToPackage{normalem}{ulem}
\usepackage{ulem}

\makeatletter
\numberwithin{equation}{section}
\numberwithin{figure}{section}
\theoremstyle{plain}
\newtheorem{thm}{\protect\theoremname}[section]
  \theoremstyle{definition}
  \newtheorem{defn}[thm]{\protect\definitionname}
  \theoremstyle{plain}
  \newtheorem{lem}[thm]{\protect\lemmaname}
  \theoremstyle{remark}
  \newtheorem{rem}[thm]{\protect\remarkname}
  \theoremstyle{plain}
  \newtheorem{cor}[thm]{\protect\corollaryname}
  \theoremstyle{plain}
  \newtheorem{question}[thm]{\protect\questionname}
  \theoremstyle{plain}
  \newtheorem{conjecture}[thm]{\protect\conjecturename}
  \theoremstyle{plain}
  \newtheorem{fact}[thm]{\protect\factname}
  \theoremstyle{remark}
  \newtheorem{notation}[thm]{\protect\notationname}
  \theoremstyle{plain}
  \newtheorem{prop}[thm]{\protect\propositionname}
  \theoremstyle{remark}
  \newtheorem*{claim*}{\protect\claimname}
  \theoremstyle{definition}
  \newtheorem{example}[thm]{\protect\examplename}
  \theoremstyle{remark}
  \newtheorem{claim}[thm]{\protect\claimname}
 \theoremstyle{remark}
 \newtheorem*{subclaim*}{\protect\subclaimname}
  \theoremstyle{definition}
  \newtheorem{problem}[thm]{\protect\problemname}
  \theoremstyle{remark}
  \newtheorem{note}[thm]{\protect\notename}

\usepackage{amsmath}
\usepackage{a4wide}
\usepackage{amssymb}
\usepackage{euler}

\usepackage{url}

\makeatother

\usepackage{babel}
  \providecommand{\claimname}{Claim}
  \providecommand{\conjecturename}{Conjecture}
  \providecommand{\corollaryname}{Corollary}
  \providecommand{\definitionname}{Definition}
  \providecommand{\examplename}{Example}
  \providecommand{\factname}{Fact}
  \providecommand{\lemmaname}{Lemma}
  \providecommand{\notationname}{Notation}
  \providecommand{\notename}{Note}
  \providecommand{\problemname}{Problem}
  \providecommand{\propositionname}{Proposition}
  \providecommand{\questionname}{Question}
  \providecommand{\remarkname}{Remark}
 \providecommand{\subclaimname}{Subclaim}
\providecommand{\theoremname}{Theorem}

\begin{document}
\global\long\def\p{\mathbf{p}}
\global\long\def\q{\mathbf{q}}
\global\long\def\C{\mathfrak{C}}
\global\long\def\Ff{\mathbb{F}}
\global\long\def\AS{\varrho}
\global\long\def\Rr{\mathbb{R}}
\global\long\def\Zz{\mathbb{Z}}
\global\long\def\Qq{\mathbb{Q}}
\global\long\def\Cc{\mathbb{C}}
 \global\long\def\pr{\operatorname{pr}}
\global\long\def\image{\operatorname{im}}
\global\long\def\otp{\operatorname{otp}}
\global\long\def\dec{\operatorname{dec}}
\global\long\def\suc{\operatorname{suc}}
\global\long\def\pre{\operatorname{pre}}
\global\long\def\qe{\operatorname{qe}}
\global\long\def\trdeg{\operatorname{tr.deg}}
 \global\long\def\ind{\operatorname{ind}}
\global\long\def\Nind{\operatorname{Nind}}
\global\long\def\lev{\operatorname{lev}}
\global\long\def\Suc{\operatorname{Suc}}
\global\long\def\HNind{\operatorname{HNind}}
\global\long\def\minb{{\lim}}
\global\long\def\concat{\frown}
\global\long\def\cl{\operatorname{cl}}
\global\long\def\tp{\operatorname{tp}}
\global\long\def\id{\operatorname{id}}
\global\long\def\cons{\left(\star\right)}
\global\long\def\qf{\operatorname{qf}}
\global\long\def\ai{\operatorname{ai}}
\global\long\def\dtp{\operatorname{dtp}}
\global\long\def\acl{\operatorname{acl}}
\global\long\def\nb{\operatorname{nb}}
\global\long\def\limb{{\lim}}
\global\long\def\leftexp#1#2{{\vphantom{#2}}^{#1}{#2}}
\global\long\def\intr{\operatorname{interval}}
\global\long\def\atom{\emph{at}}
\global\long\def\I{\mathfrak{I}}
\global\long\def\uf{\operatorname{uf}}
\global\long\def\ded{\operatorname{ded}}
\global\long\def\Ded{\operatorname{Ded}}
\global\long\def\Df{\operatorname{Df}}
\global\long\def\Th{\operatorname{Th}}
\global\long\def\eq{\operatorname{eq}}
\global\long\def\dcl{\operatorname{dcl}}
\global\long\def\supp{\operatorname{supp}}
\global\long\def\Aut{\operatorname{Aut}}
\global\long\def\ac{ac}
\global\long\def\DfOne{\operatorname{Df}_{\operatorname{iso}}}

\title{Chain conditions in dependent groups}

\author{Itay Kaplan and Saharon Shelah}

\thanks{The second author would like to thank the Israel Science Foundation
for partial support of this research. Publication no. 993 on Shelah's
list.}
\begin{abstract}
In this note we prove and disprove some chain conditions in type definable
and definable groups in dependent, strongly dependent and strongly$^{2}$
dependent theories.
\end{abstract}
\maketitle

\section{Introduction}

This note is about chain conditions in dependent, strongly dependent
and strongly$^{2}$ dependent theories. 

Throughout, all formulas will be first order, $T$ will denote a complete
first order theory, and $\C$ will be the monster model of $T$ ---
a very big saturated model that contains all small models. We do not
differentiate between finite tuples and singletons unless we state
it explicitly. 
\begin{defn}
A formula $\varphi\left(x,y\right)$ has the independence property
in some model if for every $n<\omega$ there are $\left\langle a_{i},b_{s}\left|i<n,s\subseteq n\right.\right\rangle $
such that $\varphi\left(a_{i},b_{s}\right)$ holds iff $i\in s$.

A (first order) theory $T$ is dependent (sometimes also NIP) if it
does not have the independence property: there is no formula $\varphi\left(x,y\right)$
that has the independence property in any model of $T$. A model $M$
is dependent if $Th\left(M\right)$ is.
\end{defn}
For a good introduction to dependent theories appears we recommend
\cite{Ad}, but we shall give an exact reference to any fact we use,
so no prior knowledge is assumed.

What do we mean by a chain condition? rather than giving an exact
definition, we give an example of such a condition --- the first one.
It is the Baldwin-Saxl Lemma, which we shall present with the (very
easy and short) proof.
\begin{defn}
\label{def:uniformly definable}Suppose $\varphi\left(x,y\right)$
is a formula. Then if $G$ is a definable group in some model, and
for all $c\in C$, $\varphi\left(x,c\right)$ defines a subgroup,
then $\left\{ \varphi\left(\C,c\right)\left|\, c\in C\right.\right\} $
is a family of \emph{uniformly definable subgroups}.\end{defn}
\begin{lem}
\label{lem:BaldwinSaxl}\cite{BaSxl} Let $G$ be a group definable
in a dependent theory. Suppose $\varphi\left(x,y\right)$ is a formula
and that $\left\{ \varphi\left(x,c\right)\left|\, c\in C\right.\right\} $
defines a family of subgroups of $G$. Then there is a number $n<\omega$
such that any finite intersection of groups from this family is already
an intersection of $n$ of them.\end{lem}
\begin{proof}
Suppose not, then for every $n<\omega$ there are $c_{0},\ldots,c_{n-1}\in C$
and $g_{0},\ldots,g_{n-1}\in G$ (in some model) such that $\varphi\left(g_{i},c_{j}\right)$
holds iff $i\neq j$. For $s\subseteq n$, let $g_{s}=\prod_{i\in s}g_{i}$
(the order does not matter), then $\varphi\left(g_{s},c_{j}\right)$
iff $j\notin s$ --- this is a contradiction. 
\end{proof}
In stable theories (which we shall not define here), the Baldwin-Saxl
lemma is even stronger: every intersection of such a family is really
a finite one (see \cite[Proposition 1.4]{PoizatSatbleGroups}). 

The focus of this note is type definable groups in dependent theories,
where such a proof does not work.
\begin{defn}
A \emph{type definable group} for a theory $T$ is a type --- a collection
$\Sigma\left(x\right)$ of formulas (maybe over parameters), and a
formula $\nu\left(x,y,z\right)$, such that in the monster model $\C$
of $T$, $\left\langle \Sigma\left(\C\right),\nu\right\rangle $ is
a group with $\nu$ defining the group operation (without loss of
generality, $T\models\forall xy\exists^{\leq1}z\left(\nu\left(x,y,z\right)\right)$).
We shall denote this operation by $\cdot$. 
\end{defn}
In stable theories, their analysis becomes easier as each type definable
group is an intersection of definable ones (see \cite{PoizatSatbleGroups}).
\begin{rem}
In this note we assume that $G$ is a finitary type definable group,
i.e. $x$ above is a finite tuple.\end{rem}
\begin{defn}
\label{def:unbounded index}Suppose $G\geq H$ are two type definable
groups ($H$ is a subgroup of $G$). We say that the index $\left[G:H\right]$
is \emph{unbounded}, or $\infty$, if for any cardinality $\kappa$,
there exists a model $M\models T$, such that $\left[G^{M}:H^{M}\right]\geq\kappa$.
Equivalently (by the Erd\H{o}s-Rado coloring theorem), this means
that there exists (in $\C$) a sequence of indiscernibles $\left\langle a_{i}\left|\, i<\omega\right.\right\rangle $
(over the parameters defining $G$ and $H$) such that $a_{i}\in G$
for all $i$, and $i<j\Rightarrow a_{i}\cdot a_{j}^{-1}\notin H$.
In $\C$, this means that $\left[G^{\C}:H^{\C}\right]=\left|\C\right|$.
When $G$ and $H$ are definable, then by compactness this is equivalent
to the index $\left[G:H\right]$ being infinite. 

So $\left[G:H\right]$ is \emph{bounded} if it is not unbounded.
\end{defn}
This leads to the following definition
\begin{defn}
Let $G$ be a type definable group. 
\begin{enumerate}
\item For a set $A$, $G_{A}^{00}$ is the minimal $A$-type definable subgroup
of $G$ of bounded index. 
\item We say that $G^{00}$ exists if $G_{A}^{00}=G_{\emptyset}^{00}$ for
all $A$.
\end{enumerate}
\end{defn}
Shelah proved
\begin{thm}
\cite{Sh876} \emph{If $G$ is a type definable group in a dependent
theory, then $G^{00}$ exists. }
\end{thm}
Even though fields are not the main concern of this note, the following
question is in the basis of its motivation. Recall
\begin{thm}
\cite[Theorem VI.6.4]{Lang} (Artin-Schreier) Let $k$ be a field
of characteristic $p$. Let $\AS$ be the polynomial $X^{p}-X$. 
\begin{enumerate}
\item Given $a\in k$, either the polynomial $\AS-a$ has a root in $k$,
in which case all its root are in $k$, or it is irreducible. In the
latter case, if $\alpha$ is a root then $k(\alpha)$ is cyclic of
degree $p$ over $k$.
\item Conversely, let $K$ be a cyclic extension of $k$ of degree $p$.
Then there exists $\alpha\in K$ such that $K=k(\alpha)$ and for
some $a\in k$, $\AS(\alpha)=a$.
\end{enumerate}

Such extensions are called Artin-Schreier extensions. 

\end{thm}
The first author, in a joint paper with Thomas Scanlon and Frank Wagner,
proved
\begin{thm}
\label{thm:NIPFieldsAS}\cite{KaScWa} Let $K$ be an infinite dependent
field of characteristic $p>0$. Then $K$ is Artin-Schreier closed
--- i.e. $\AS$ is onto. 
\end{thm}
What about the type definable case? What if $K$ is an infinite type
definable field? 

In simple theories (which we shall not define), we have:
\begin{thm}
\cite{KaScWa} Let $K$ be a type definable field in a simple theory.
Then $K$ has boundedly many AS extensions. 
\end{thm}
But for the dependent case we only proved
\begin{thm}
\cite{KaScWa} For an infinite type definable field $K$ in a dependent
theory there are either unboundedly many Artin-Schreier extensions,
or none.
\end{thm}
from these two we conclude
\begin{cor}
If $T$ is stable (so it is both simple and dependent), then type
definable fields are AS closed.
\end{cor}
The following, then, is still open
\begin{question}
What about the dependent case? In other words, is it true that infinite
type definable fields in dependent theories are AS-closed?
\end{question}
Observing the proof of Theorem \ref{thm:NIPFieldsAS}, we see that
it is enough to find a number $n$, and $n+1$ algebraically independent
elements, $\left\langle a_{i}\left|\, i\leq n\right.\right\rangle $
in $k:=K^{p^{\infty}}$, such that $\bigcap_{i<n}a_{i}\AS\left(K\right)=\bigcap_{i\leq n}a_{i}\AS\left(K\right)$.
So the Baldwin-Saxl applies in the case where the field $K$ is definable.
If $K$ is type definable, we may want something similar. But what
can we prove?

A conjecture of Frank Wagner is the main motivation question
\begin{conjecture}
\label{conj:Wagner Conju}Suppose $T$ is dependent, then the following
holds
\begin{itemize}
\item [\smiley]Suppose $G$ is a type definable group. Suppose $p\left(x,y\right)$
is a type and $\left\langle a_{i}\left|\, i<\omega\right.\right\rangle $
is an indiscernible sequence such that $G_{i}=p\left(x,a_{i}\right)\leq G$.
Then there is some $n$, such that for all finite sets, $v\subseteq\omega$,
the intersection $\bigcap_{i\in v}G_{i}$ is equal to a sub-intersection
of size $n$.
\end{itemize}
\end{conjecture}
Let refer to $\smiley$ as \emph{Property A} (of a theory $T$) for
the rest of the paper. So we have
\begin{fact}
If Property A is true for a theory $T$, then type definable fields
are Artin-Schreier closed.
\end{fact}
In Section \ref{sec:Strongly^{2}  dependent theories}, we deal with
strongly$^{2}$ dependent theories (this is a much stronger condition
than merely dependence), and among other things, prove that Property
A is true for them. 

In Section \ref{sec:Baldwin-Saxl-type-lemmas}, we give some generalizations
and variants of Baldwin-Saxl for type definable groups in dependent
and strongly dependent theories (which we define below). One of them
is joint work with Frank Wagner. We prove that Property A holds for
theories with bounded dp-rank.

In Section \ref{sec:An-example}, we provide a counterexample that
shows that property A does not hold in stable theories, so Conjecture
\ref{conj:Wagner Conju} as it is stated is false. 
\begin{question}
Does Property A hold for strongly dependent theories?
\end{question}

\section{\label{sec:Strongly^{2}  dependent theories}Strongly$^{2}$ dependent
theories}
\begin{notation}
We call an array of elements (or tuples) $\left\langle a_{i,j}\left|\, i,j<\omega\right.\right\rangle $
an \emph{indiscernible array} over $A$ if for $i_{0}<\omega$, the
$i_{0}$-row $\left\langle a_{i_{0},j}\left|\, j<\omega\right.\right\rangle $
is indiscernible over the rest of the sequence ($\left\{ a_{i,j}\left|\, i\neq i_{0},i,j<\omega\right.\right\} $)
and $A$, i.e. when the rows are mutually indiscernible. \end{notation}
\begin{defn}
\label{def:strongly2 dependent}A theory $T$ is said to be \uline{not}\emph{
strongly$^{2}$ dependent} if there exists a sequence of formulas
$\left\langle \varphi_{i}\left(x,y_{i},z_{i}\right)\left|\, i<\omega\right.\right\rangle $,
an array $\left\langle a_{i,j}\left|\, i,j<\omega\right.\right\rangle $
and $b_{k}\in\left\{ a_{i,j}\left|\, i<k,j<\omega\right.\right\} $
such that 
\begin{itemize}
\item The array $\left\langle a_{i,j}\left|\, i,j<\omega\right.\right\rangle $
is an indiscernible array (over $\emptyset$). 
\item The set $\left\{ \varphi_{i}\left(x,a_{i,0},b_{i}\right)\land\neg\varphi_{i}\left(x,a_{i,1},b_{i}\right)\left|\, i<\omega\right.\right\} $
is consistent.
\end{itemize}
So $T$ is \emph{strongly$^{2}$ dependent} when this configuration
does not exist.

Note that the roles of $i$ and $j$ are not symmetric.
\end{defn}
(In the definition above, $x,z_{i},y_{i}$ can be tuples, the length
of $z_{i}$ and $y_{i}$ may depend on $i$).

This definition was introduced and discussed in \cite{Sh863} and
\cite{Sh783}.
\begin{rem}
\label{rem:singleton}By \cite[Claim 2.8]{Sh863}, we may assume in
the definition above that $x$ is a singleton.\end{rem}
\begin{prop}
\label{prop:noInfiniteChain}Suppose $T$ is strongly$^{2}$ dependent,
then it is impossible to have a sequence of type definable groups
$\left\langle G_{i}\left|\, i<\omega\right.\right\rangle $ such that
$G_{i+1}\leq G_{i}$ and $\left[G_{i}:G_{i+1}\right]=\infty$.\end{prop}
\begin{proof}
Without loss of generality, we shall assume that all groups are definable
over $\emptyset$. Suppose there is such a sequence $\left\langle G_{i}\left|\, i<\omega\right.\right\rangle $.
Let $\left\langle a_{i,j}\left|\, i,j<\omega\right.\right\rangle $
be an indiscernible array such that for each $i<\omega$, the sequence
$\left\langle a_{i,j}\left|\, j<\omega\right.\right\rangle $ is a
sequence from $G_{i}$ (in $\C$) such that $a_{i,j'}^{-1}\cdot a_{i,j}\notin G_{i+1}$
for all $j<j'<\omega$. We can find such an array because of our assumption
and Ramsey (for more detail, see the proof of Corollary \ref{cor:OpenQuestionStrong2}
below).

For each $i<\omega$, let $\psi_{i}\left(x\right)$ be in the type
defining $G_{i+1}$ such that $\neg\psi_{i}\left(a_{i,j'}^{-1}\cdot a_{i,j}\right)$.
By compactness, there is a formula $\xi_{i}\left(x\right)$ in the
type defining $G_{i+1}$ such that for all $a,b\in\C$, if $\xi_{i}\left(a\right)\land\xi_{i}\left(b\right)$
then $\psi_{i}\left(a\cdot b^{-1}\right)$ holds. Let $\varphi_{i}\left(x,y,z\right)=\xi_{i}\left(y^{-1}\cdot z^{-1}\cdot x\right)$.
For $i<\omega$, let $b_{i}=a_{0,0}\cdot\ldots\cdot a_{i-1,0}$ (so
$b_{0}=1$). 

Let us check that the set $\left\{ \varphi_{i}\left(x,a_{i,0},b_{i}\right)\land\neg\varphi_{i}\left(x,a_{i,1},b_{i}\right)\left|\, i<\omega\right.\right\} $
is consistent. Let $i_{0}<\omega$, and let $c=b_{i_{0}}$. Then for
$i<i_{0}$, $\varphi_{i}\left(c,a_{i,0},b_{i}\right)$ holds iff $\xi_{i}\left(a_{i+1,0}\cdot\ldots\cdot a_{i_{0}-1,0}\right)$
but the product $a_{i+1,0}\cdot\ldots\cdot a_{i_{0}-1,0}$ is an element
of $G_{i+1}$ and $\xi_{i}$ is in the type defining $G_{i+1}$, so
$\varphi_{i}\left(c,a_{i,0},b_{i}\right)$ holds. Now, $\varphi_{i}\left(c,a_{i,1},b_{i}\right)$
holds iff $\xi_{i}\left(a_{i,1}^{-1}a_{i,0}\cdot\ldots\cdot a_{i_{0}-1,0}\right)$.
However, since $\xi_{i}\left(a_{i+1,0}\cdot\ldots\cdot a_{i_{0}-1,0}\right)$
holds, by choice of $\xi_{i}$ we have 
\[
\psi_{i}\left(\left[a_{i,1}^{-1}a_{i,0}\cdot\ldots\cdot a_{i_{0}-1,0}\right]\cdot\left[a_{i+1,0}\cdot\ldots\cdot a_{i_{0}-1,0}\right]^{-1}\right)
\]
 i.e. $\psi_{i}\left(a_{i,1}^{-1}\cdot a_{i,0}\right)$ holds ---
contradiction. \end{proof}
\begin{rem}
It is well known (see \cite{PoizatSatbleGroups}) that in superstable
theories the same proposition hold.
\end{rem}
The next corollary already appeared in \cite[Claim 0.1]{Sh863} with
definable groups instead of type definable (with proof already in
\cite[Claim 3.10]{Sh783}). 
\begin{cor}
\label{cor:finiteIndexRange} Assume $T$ is strongly$^{2}$ dependent.
If $G$ is a type definable group and $h$ is a definable homomorphism
$h:G\to G$ with finite kernel then $h$ is almost onto $G$, i.e.,
the index $\left[G:h\left(G\right)\right]$ is bounded (i.e. $<\infty$).
If $G$ is definable, then the index must be finite.\end{cor}
\begin{proof}
Consider the sequence of groups $\left\langle h^{\left(i\right)}\left(G\right)\left|\, i<\omega\right.\right\rangle $
(i.e. $G$, $h\left(G\right)$, $h\left(h\left(G\right)\right)$,
etc.). By Proposition \ref{prop:noInfiniteChain}, for some $i<\omega$,
$\left[h^{\left(i\right)}\left(G\right):h^{\left(i+1\right)}\left(G\right)\right]<\infty$.
Now the Corollary easily follows from
\begin{claim*}
If $G$ is a group, $h:G\to G$ a homomorphism with finite kernel,
then $\left[G:h\left(G\right)\right]+\aleph_{0}=\left[h\left(G\right):h\left(h\left(G\right)\right)\right]+\aleph_{0}$. \end{claim*}
\begin{proof}
\renewcommand{\qedsymbol}{}(of claim) Let $H=h\left(G\right)$. Easily,
one has $\left[H:h\left(H\right)\right]\leq\left[G:H\right]$. 

We may assume that $\left[G:H\right]$ is infinite. Let $\ker\left(h\right)=\left\{ g_{0},\ldots,g_{k-1}\right\} $.
Suppose that $\left[G:H\right]=\kappa$ but $\left[H:h\left(H\right)\right]<\kappa$.
So let $\left\{ a_{i}\left|\, i<\kappa\right.\right\} \subseteq G$
are such that $a_{i}^{-1}\cdot a_{j}\notin h\left(G\right)$ for $i\neq j$.
So there must be some coset $a\cdot h\left(H\right)$ in $H$ such
that for infinitely many $i<\kappa$, $h\left(a_{i}\right)\in a\cdot h\left(H\right)$.
Let us enumerate them as $\left\langle a_{i}\left|\, i<\omega\right.\right\rangle $.
So for $i<j<\omega$, let $C\left(a_{i},a_{j}\right)$ be the least
number $l<k$ such that there is some $y\in h\left(G\right)$ with
$y^{-1}a_{i}^{-1}a_{j}=g_{l}$. By Ramsey, we may assume that $C\left(a_{i},a_{j}\right)$
is constant. Now pick $i_{1}<i_{2}<j<\omega$. So we have $y^{-1}a_{i_{1}}^{-1}a_{j}=\left(y'\right)^{-1}a_{i_{2}}^{-1}a_{j}$,
so $y^{-1}a_{i_{1}}^{-1}=\left(y'\right)^{-1}a_{i_{2}}^{-1}$ and
hence $a_{i_{1}}^{-1}a_{i_{2}}=y\left(y'\right)^{-1}\in h\left(G\right)$
--- contradiction.
\end{proof}
\end{proof}
\begin{cor}
\label{cor:finiteIndex}If $K$ is a strongly$^{2}$ dependent field,
(or even a type definable field in a strongly$^{2}$ dependent theory)
then for all $n<\omega$, $\left[K^{\times}:\left(K^{\times}\right)^{n}\right]<\infty$. 
\end{cor}

\begin{cor}
\label{cor:OpenQuestionStrong2}Let $G$ be type definable group in
a strongly$^{2}$ dependent theory $T$. \end{cor}
\begin{enumerate}
\item Given a family of uniformly type definable subgroups $\left\{ p\left(x,a_{i}\right)\left|\, i<\omega\right.\right\} $
such that $\left\langle a_{i}\left|\, i<\omega\right.\right\rangle $
is an indiscernible sequence, there is some $n<\omega$ such that
$\bigcap_{j<\omega}p\left(\C,a_{j}\right)=\bigcap_{j<n}p\left(\C,a_{j}\right)$.
In particular, $T$ has Property A. 
\item Given a family of uniformly definable subgroups $\left\{ \varphi\left(x,c\right)\left|\, c\in C\right.\right\} $,
the intersection 
\[
\bigcap_{c\in C}\varphi\left(\C,c\right)
\]
 is already a finite one. \end{enumerate}
\begin{proof}
(1) Assume without loss of generality that $G$ is defined over $\emptyset$.
Let $G_{i}=p\left(\C,a_{i}\right)$, and let $H_{i}=\bigcap_{j<i}G_{i}$.
By Proposition \ref{prop:noInfiniteChain}, for some $i_{0}<\omega$,
$\left[H_{i_{0}}:H_{i_{0}+1}\right]<\infty$. For $r\geq i_{0}$,
let $H_{i_{0},r}=\bigcap_{j<i_{0}}G_{j}\cap G_{r}$ (so $H_{i_{0}+1}=H_{i_{0},i_{0}}$).
By indiscerniblity, $\left[H_{i_{0}}:H_{i_{0},r}\right]<\infty$.
This means (by definition of $H_{i_{0}}^{00}$) that $H_{i_{0}}^{00}\leq H_{i_{0},r}$
for all $r>i_{0}$. However, if $H_{i_{0},i_{0}}\neq H_{i_{0},r}$
for some $i_{0}<r<\omega$, then by indiscerniblity $H_{i_{0},r}\neq H_{i_{0},r'}$
for all $i_{0}\leq r<r'$, and by compactness and indiscerniblity
we may increase the length $\omega$ of the sequence to any cardinality
$\kappa$, so that the size of $H_{i_{0}}/H_{i_{0}}^{00}$ is unbounded
--- contradiction. This means that $H_{i_{0}+1}\subseteq G_{r}$ for
all $r>i_{0}$, and so $\bigcap_{i<\omega}G_{i}=\bigcap_{i<i_{0}+1}G_{i}$.

(2) Assume not. Then we can find a sequence $\left\langle c_{i}\left|\, i<\omega\right.\right\rangle $
of element of $C$ such that $\bigcap_{j<i}\varphi\left(\C,c_{j}\right)\neq\bigcap_{j<i+1}\varphi\left(\C,c_{j}\right)$.
By Ramsey, we can extract an indiscernible sequence $\left\langle a_{i}\left|\, i<\omega\right.\right\rangle $
such that for any $n$, and any formula $\psi\left(x_{0},\ldots,x_{n-1}\right)$,
if $\psi\left(a_{0},\ldots,a_{n-1}\right)$ holds then there are $i_{0}<\ldots<i_{n-1}$
such that $\psi\left(c_{i_{0}},\ldots,c_{i_{n-1}}\right)$ holds.
In particular, $\varphi\left(\C,a_{i}\right)$ defines a subgroup
of $G$ and $\bigcap_{j<i}\varphi\left(\C,a_{j}\right)\neq\bigcap_{j<i+1}\varphi\left(\C,a_{j}\right)$.
But this contradicts (1). 
\end{proof}
As further applications, we show that some theories are not strongly$^{2}$
dependent.
\begin{example}
\label{exa:OAG}Suppose $\left\langle G,+,<\right\rangle $ is an
ordered abelian group. Then its theory $Th\left(G,+,0,<\right)$ is
not strongly$^{2}$ dependent.\end{example}
\begin{proof}
We work in the monster model $\C$. Let $G_{d}=\left\{ x\in\C\left|\,\forall n<\omega\left(n\mid x\right)\right.\right\} $,
so it is a divisible ordered subgroup of $G$. Note that since $G$
is ordered, it is torsion free, so really it is a $\Qq$-vector space.
Define a descending sequence of infinite type definable groups $G_{d}^{i}\leq G_{d}$
for $i<\omega$ such that $\left[G_{d}^{i}:G_{d}^{i+1}\right]=\infty$.
This contradicts Proposition \ref{prop:noInfiniteChain}. Let $G_{d}^{0}=G_{d}$,
and suppose we have chosen $G_{d}^{i}$. Let $a_{i}\in G_{d}^{i}$
be positive. Let $G_{d}^{i+1}=G_{d}^{i}\cap\bigcap_{n<\omega}\left(-a_{i}/n,a_{i}/n\right)$.
This is a type definable subgroup of $G_{d}^{i}$. The sequence $\left\langle k\cdot a_{i}\left|\, k<\omega\right.\right\rangle $
satisfies $\left(k-l\right)\cdot a_{i}\notin\left(-a_{i}/2,a_{i}/2\right)$
for any $k\neq l$, and by Ramsey (as in the proof of Corollary \ref{cor:OpenQuestionStrong2}
(2)) we get $\left[G_{d}^{i}:G_{d}^{i+1}\right]=\infty$. 
\end{proof}

\begin{example}
The theory $Th\left(\Rr,+,\cdot,0,1\right)$ is strongly dependent
(it is even o-minimal, so dp-minimal --- see Definitions \ref{def:dp-rank}
and \ref{def:strongly-dep} below). However it is not strongly$^{2}$
dependent. 
\end{example}

\begin{example}
The theory $Th\left(\Qq_{p},+,\cdot,0,1\right)$ of the p-adics is
strongly dependent (it is also dp-minimal), but not strongly$^{2}$
dependent: The valuation group $\left(\Zz,+,0,<\right)$ is interpretable.
\end{example}
Adding some structure to an algebraically closed field, we can easily
get a strongly$^{2}$ dependent theory.
\begin{example}
Let $L=L_{\mbox{rings}}\cup\left\{ P,<\right\} $ where $L_{\mbox{rings}}$
is the language of rings $\left\{ +,\cdot,0,1\right\} $, $P$ is
a unary predicate and $<$ is a binary relation symbol. Let $K$ be
$\Cc$ (so is an algebraically closed field), and let $P\subseteq K$
be a countable set of algebraically independent elements, enumerated
as $\left\{ a_{i}\left|\, i\in\Qq\right.\right\} $. Let $M=\left\langle K,P,<\right\rangle $
where $a<^{M}b$ iff $a,b\in P$ and $a=a_{i},\, b=a_{j}$ where $i<j$.
Let $T=Th\left(M\right)$.\end{example}
\begin{claim}
$T$ is strongly$^{2}$ dependent.\end{claim}
\begin{proof}
Note that $T$ is axiomatizable by saying that the universe is an
algebraically closed field, $P$ is a subset of algebraically independent
elements and $<$ is a dense linear order on $P$ (to see this, take
two saturated models of the same size and show that they are isomorphic). 

Let us fix some terminology: 
\begin{itemize}
\item When we write $\acl$, we mean the algebraic closure in the field
sense. When we say basis, we mean a transcendental basis.
\item When we say that a set is independent / dependent over $A$ for some
set $A$, we mean that it is dependent / independent in the pregeometry
induced by $\cl\left(X\right)=\acl\left(AX\right)$.
\item $\dcl\left(X\right)$ stands for the definable closure of $X$. 
\end{itemize}
We work in a saturated model $\C$ of $T$. 

Suppose $X$ is some set. Let $X_{0}$ be some basis for $X$ over
$P$, and let $\dcl^{P}\left(X\right)$ be the set of $p\in P$ such
that there exists some minimal finite $P_{0}\subseteq P$ with $p\in P_{0}$
and some $x\in X$ such that $x\in\acl\left(P_{0}X_{0}\right)$. Note
that this set is contained in $\dcl\left(X\right)$ (since $P$ is
linearly ordered) and that it does not depend on the choice of $X_{0}$. 

Suppose $a$ is a finite tuple, and $A$ is a set. Let $A^{P}=\dcl^{P}\left(A\right)$. 

Let $\tp_{K}\left(a/A\right)$ the type of $a\concat\left(Aa\right)^{P}$
(considered as a tuple, ordered by $<^{\C}$) over $A\cup A^{P}$
in the field language, and $\tp_{P}\left(a/A\right)$ the type of
the tuple $\left(Aa\right)^{P}$ over $A^{P}$ in the order language.
\begin{subclaim*}
For finite tuples $a,b$ and a set $A$, $\tp\left(a/A\right)=\tp\left(b/A\right)$
iff $\tp_{P}\left(Aa/A\right)=\tp_{P}\left(Ab/A\right)$ and $\tp_{K}\left(a/A\right)=\tp_{K}\left(b/A\right)$.
In fact, in this case, there is an automorphism of the field $\acl\left(abAP\right)$
fixing $A$ pointwise and $P$ setwise taking $a$ to $b$. This automorphism
is an elementary map.\end{subclaim*}
\begin{proof}
Given that the $P$ and $K$ types are equal, it is easy to construct
an automorphism of $\acl\left(abAP\right)$ as above. First we construct
an automorphism of $\left\langle P,<\right\rangle $ that takes $a^{P}$
to $b^{P}$ and fixes $A^{P}$. We can extend this automorphism to
$A_{0}P$ where $A_{0}$ is a basis of $A$ over $P$. By definition
of $\dcl^{P}$, we can also extend it to $\acl\left(AP\right)$, fixing
$A$ pointwise. Let $a'\subseteq a$ be a basis for $a$ over $AP$,
and $b'\subseteq b$ a basis for $b$ over $AP$. By definition of
$\dcl^{P}$, $\left|a'\right|=\left|b'\right|$. This means we can
extend this automorphism to $\acl\left(aAP\right)$, taking it to
$\acl\left(bAP\right)$. Next extend this to an automorphism of $\acl\left(abAP\right)$
(possible since both $a$ and $b$ are finite). Now we can extend
this to an automorphism of $\C$ since it is algebraically closed.
Note that if $c\notin\acl\left(abAP\right)$, we can choose this automorphism
to fix $c$.
\end{proof}
Suppose that $\left\langle a_{i,j}\left|\, i,j<\omega\right.\right\rangle $
is an indiscernible array over a parameter set $A$ as in Definition
\ref{def:strongly2 dependent} and that $c$ is a singleton such that: 
\begin{itemize}
\item The sequence $I_{0}:=\left\langle a_{0,j}\left|\, j<\omega\right.\right\rangle $
is not indiscernible over $c$, and moreover $\tp\left(a_{0,0}/c\right)\neq\tp\left(a_{0,1}/c\right)$.
\item For $i>0$, the sequence $I_{i}:=\left\langle a_{i,j}\left|\, j<\omega\right.\right\rangle $
is not indiscernible over $c\cup\bigcup_{k<i}I_{k}\cup A$.
\end{itemize}
Suppose first that $c\notin\acl\left(APa_{0,0}a_{0,1}\right)$. Then,
by the proof of the first subclaim, we get a contradiction, since
there is an automorphism fixing $cA$ pointwise and $P$ setwise taking
$a_{0,0}$ to $a_{0,1}$. So $c\in\acl\left(APa_{0,0}a_{0,1}\right)$.
Increase the parameter set $A$ by adding the first row $\left\langle a_{0,j}\left|\, j<\omega\right.\right\rangle $.
So we may assume that $c\in\acl\left(AP\right)$. Choose a basis $A_{0}\subseteq A$
over $P$, and let $c^{P}\subseteq P$ the unique minimal tuple of
elements such that $c\in\acl\left(A_{0}c^{P}\right)$. Since $c\in\acl\left(Ac^{P}\right)$,
we may replace $c$ by $c^{P}$ and assume that $c$ is a tuple of
elements in $P$ (here we use the fact that if $I$ is indiscernible
over $Ac^{P}$ then it is also indiscernible over $\acl\left(Ac^{P}\right)$). 

Expand all the sequences to order type $\omega^{*}+\omega+\omega$.
Let $B=\bigcup\left\{ a_{i,j}\left|\, i<\omega,j<0\vee\omega\leq j\right.\right\} \cup A$.
For each $i<\omega$ and $0\leq j<\omega$, let $a_{i,j}^{P}$ be
$\dcl^{P}\left(a_{i,j}B\right)$ considered as a tuple ordered by
$<^{\C}$, and let $B^{P}=\dcl^{P}\left(B\right)$. Then $\left\langle a_{i,j}^{P}\left|\, i,j<\omega\right.\right\rangle $
is an indiscernible array over $B^{P}$ and $\left\langle a_{i,j}\concat a_{i,j}^{P}\left|\, i,j<\omega\right.\right\rangle $
is an indiscernible array over $B\cup B^{P}$. 

As both the theories of dense linear orders and algebraically closed
fields are strongly$^{2}$ dependent (this is easy to check), there
is some $i_{0}$ such that $\left\langle a_{i_{0},j}^{P}\left|\, j<\omega\right.\right\rangle $
is indiscernible over $cB^{P}\cup\left\{ a_{i,j}^{P}\left|\, i<i_{0},j<\omega\right.\right\} $
in the order language and $\left\langle a_{i_{0},j}\concat a_{i_{0},j}^{P}\left|\, j<\omega\right.\right\rangle $
is indiscernible over $cB\cup B^{P}\cup\left\{ a_{i,j}\concat a_{i,j}^{P}\left|\, i<i_{0},j<\omega\right.\right\} $
in the field language. 

Let $C=\bigcup\left\{ a_{i,j}\left|\, i<i_{0},j<\omega\right.\right\} $.
We must check that $\left\langle a_{i_{0},j}\left|\, j<\omega\right.\right\rangle $
is indiscernible over $BCc$. Let us show, for instance, that $\tp\left(a_{i_{0},0}/BCc\right)=\tp\left(a_{i_{0},1}/BCc\right)$.
For this we apply the subclaim. We claim that $\dcl^{P}\left(BCc\right)=\bigcup\left\{ a_{i,j}^{P}\left|\, i<i_{0},j<\omega\right.\right\} \cup B^{P}\cup c$.
Why? Choose some basis $D$ for $BC$ over $P$ such that $D$ contains
a basis for $B$ over $P$. If some element $x$ in $C$ is in $\acl\left(DP\right)$,
then by indiscerniblity, $x\in\acl\left(\left(a_{i,j}\cap D\right)\cup BP\right)$
for some $i,j$, which means that $x\in\acl\left(P\cup\left(\left(a_{i,j}B\right)\cap D\right)\right)$,
so the tuple from $P$ that witnesses this is already in $a_{i,j}^{P}$.
Similarly, $\dcl^{P}\left(a_{i_{0},j}BCc\right)=a_{i_{0},j}^{P}\cup\dcl^{P}\left(BC\right)\cup c$.
By the subclaim above, we are done.\end{proof}
\begin{rem}
With the same proof, one can show that if $T$ is strongly minimal,
and $P=\left\{ a_{i}\left|\, i<\omega\right.\right\} $ is an infinite
indiscernible set in $M\models T$ of cardinality $\aleph_{1}$, the
theory of the structure $\left\langle M,P,<\right\rangle $ where
$<$ is some dense linear order with no end points on $P$, is strongly$^{2}$
dependent. 
\end{rem}
We finish this section with the following conjecture:
\begin{conjecture}
All strongly$^{2}$ dependent groups are stable, i.e. if $G$ is a
group such that $Th\left(G,\cdot\right)$ is strongly$^{2}$ dependent,
then it is stable.
\end{conjecture}
Example \ref{exa:OAG} and Corollary \ref{cor:OpenQuestionStrong2}
show that this might be reasonable. This is related to the conjecture
of Shelah in \cite{Sh863} that all strongly$^{2}$ dependent infinite
fields are algebraically closed.

\section{\label{sec:Baldwin-Saxl-type-lemmas}Baldwin-Saxl type lemmas}

The next lemma is the type definable version of the Baldwin-Saxl Lemma
(see Lemma \ref{lem:BaldwinSaxl}). But first,
\begin{notation}
\label{not:Size of a type}If $p\left(x,y\right)$ is a partial type,
then $\left|p\right|$ is the size of the set of formulas $\varphi\left(x,z_{1},\ldots,z_{n}\right)$
(where $z_{i}$ is a singleton) such that for some finite tuple $y_{1},\ldots,y_{n}\in y$,
$\varphi\left(x,y_{1},\ldots,y_{n}\right)\in p$. In this sense, the
size of any type is bounded by $\left|T\right|$. \end{notation}
\begin{lem}
\label{lem:BaldSaxType}Suppose $G$ is a type definable group in
a dependent theory $T$.
\begin{enumerate}
\item If $p_{i}\left(x,y_{i}\right)$ is a type of for $i<\kappa$ ($y_{i}$
may be an infinite tuple), $\left|\bigcup p_{i}\right|<\kappa$, and
$\left\langle c_{i}\left|\, i<\kappa\right.\right\rangle $ is a sequence
of tuples such that $p_{i}\left(\C,c_{i}\right)$ is a subgroup of
$G$, then for some $i_{0}<\kappa$, $\bigcap_{i<\kappa}p_{i}\left(\C,c_{i}\right)=\bigcap_{i<\kappa,i\neq i_{0}}p_{i}\left(\C,c_{i}\right)$.
\item In particular, Given a family of uniformly type definable subgroups,
defined by $p\left(x,y\right)$, and $C$ of size $\left|p\right|^{+}$,
there is some $c_{0}\in C$ such that $\bigcap_{c\neq c_{0}}p\left(\C,c\right)=\bigcap_{c\in C}p\left(\C,c\right)$.
\item In particular, if $\left\{ G_{i}\left|\, i<\left|T\right|^{+}\right.\right\} $
is a family of type definable subgroups (defined with parameters),
then there is some $i_{0}<\left|T\right|^{+}$ such that $\bigcap G_{i}=\bigcap_{i\neq i_{0}}G_{i}$.
\end{enumerate}
\end{lem}
\begin{proof}
(1) Denote $H_{i}=p_{i}\left(\C,c_{i}\right)$. Suppose not, i.e.
for all $i<\kappa$, there is some $g_{i}$ such that $g_{i}\in H_{j}$
iff $i\neq j$. If $d_{1},d_{2}\in H_{i}$ then $d_{1}\cdot g_{i}\cdot d_{2}\notin H_{i}$.
Hence by compactness there is some formula $\varphi_{i}$, $\varphi_{i}\left(x,c_{i}\right)\in p_{i}\left(x,c_{i}\right)$
 such that for all such $d_{1},d_{2}\in H_{i}$, $\neg\varphi_{i}\left(d_{1}g_{i}d_{2},c_{i}\right)$
holds. Since $\left|\bigcup p_{i}\right|<\kappa$, we may assume that
for $i<\omega$, $\varphi_{i}$ is constant and equals $\varphi\left(x,y\right)$.
Now for any finite subset $s\subseteq\omega$, let $g_{s}=\prod_{i\in s}g_{i}$
(the order does not matter). So we have $\varphi\left(g_{s},c_{i}\right)$
iff $i\notin s$ --- a contradiction. 

(2) and (3) now follow easily from (1). 
\end{proof}
In (2) of Lemma \ref{lem:BaldSaxType}, if $C$ is an indiscernible
sequence, then the situation is simpler:
\begin{cor}
Suppose $G$ is a type definable group in a dependent theory $T$.
Given a family of uniformly type definable subgroups, defined by $p\left(x,y\right)$,
and an indiscernible sequence $C=\left\langle a_{i}\left|\, i\in\mathbb{Z}\right.\right\rangle $,
$\bigcap_{i\neq0}p\left(\C,a_{i}\right)=\bigcap_{i\in\mathbb{Z}}p\left(\C,a_{i}\right)$.\end{cor}
\begin{proof}
Assume not. By indiscernibility, we get that for all $i\in\mathbb{Z}$,
$\bigcap_{j\neq i}p\left(\C,a_{j}\right)=p\left(\C,a_{i}\right)$.
Let $I$ be an indiscernible sequence which extends $C$ to length
$\left|p\right|^{+}$. Then by indiscernibility and compactness the
same is true for this sequence. This contradicts Lemma \ref{lem:BaldSaxType}.\end{proof}
\begin{rem}
In the proof that $G^{00}$ exists in dependent theories, the above
corollary is in the kernel of the proof.
\end{rem}
If $T$ is strongly dependent, and $C$ is indiscernible, we can even
assume that the order type is $\omega$. Let us recall,
\begin{defn}
\label{def:strongly-dep}A theory $T$ is said to be \emph{\uline{not}}\emph{
strongly dependent} if there exists a sequence of formulas $\left\langle \varphi_{i}\left(x,y_{i}\right)\left|\, i<\omega\right.\right\rangle $
and an array $\left\langle a_{i,j}\left|\, i,j<\omega\right.\right\rangle $
such that 
\begin{itemize}
\item The array $\left\langle a_{i,j}\left|\, i,j<\omega\right.\right\rangle $
is an indiscernible array (over $\emptyset$). 
\item The set $\left\{ \varphi_{i}\left(x,a_{i,0}\right)\land\neg\varphi_{i}\left(x,a_{i,1}\right)\left|\, i<\omega\right.\right\} $
is consistent.
\end{itemize}
So $T$ is \emph{strongly dependent} when this configuration does
not exist.\end{defn}
\begin{lem}
\label{lem:strong1BalSaxl}Suppose $G$ is a type definable group
in a strongly dependent theory $T$. Given a family of type definable
subgroups $\left\{ p_{i}\left(x,a_{i}\right)\left|\, i<\omega\right.\right\} $
such that $\left\langle a_{i}\left|\, i<\omega\right.\right\rangle $
is an indiscernible sequence and $p_{2i}=p_{2i+1}$ for all $i<\omega$,
there is some $i<\omega$ such that $\bigcap_{j\neq i}p_{j}\left(\C,a_{j}\right)=\bigcap_{j<\omega}p_{j}\left(\C,a_{j}\right)$.

In particular, this is true when $p$ is constant. \end{lem}
\begin{proof}
Denote $H_{i}=p_{i}\left(\C,a_{i}\right)$. Assume not, i.e. for all
$i<\omega$, there exists some $g_{i}\in G$ such that $g_{i}\in H_{j}$
iff $i\neq j$. For each even $i<\omega$ we find a formula $\varphi_{i}\left(x,y\right)\in p_{i}\left(x,y\right)$
such that for all $d_{1},d_{2}\in H_{i}$, $\neg\varphi_{i}\left(d_{1}g_{i}d_{2},a_{i}\right)$.
Let $n<\omega$, and consider the product $g_{n}=\prod_{i<n,\,2\mid i}g_{i}$
(the order does not matter). Then for odd $i<n$, $\varphi_{i-1}\left(g_{n},a_{i}\right)$
holds (because $\varphi_{i-1}\in p_{i-1}=p_{i}$ by assumption), and
for even $i<n$, $\neg\varphi_{i}\left(g_{n},a_{i}\right)$ holds.
By compactness, we can find $g\in G$ such that $\varphi_{i-1}\left(g_{n},a_{i}\right)$
holds for all odd $i<\omega$ and $\neg\varphi_{i}\left(g,a_{i}\right)$
for all even $i<\omega$. Now expand the sequence by adding a sequence
$\left\langle b_{i,j}\left|\, j<\omega\right.\right\rangle $ after
each pair $a_{2i},a_{2i+1}$. Then the array defined by $a_{i,0}=a_{2i}$,
$a_{i,1}=a_{2i+1}$ and $a_{i,j}=b_{i,j-2}$ for $j\geq2$ will show
that the theory is not strongly dependent. 
\end{proof}
If the theory is of bounded dp-rank, then we can say even more.
\begin{defn}
\label{def:dp-rank}A theory $T$ is said to have \emph{bounded dp-rank},
if there is some $n<\omega$ such that the following configuration
does \uline{not} exist: a sequence of formulas $\left\langle \varphi_{i}\left(x,y_{i}\right)\left|\, i<n\right.\right\rangle $
where $x$ is a \uline{singleton} and an array $\left\langle a_{i,j}\left|\, i<n,j<\omega\right.\right\rangle $
such that 
\begin{itemize}
\item The array $\left\langle a_{i,j}\left|\, i<n,j<\omega\right.\right\rangle $
is an indiscernible array (over $\emptyset$). 
\item The set $\left\{ \varphi_{i}\left(x,a_{i,0}\right)\land\neg\varphi_{i}\left(x,a_{i,1}\right)\left|\, i<n\right.\right\} $
is consistent.
\end{itemize}
$T$ is \emph{dp-minimal} if $n=2$.
\end{defn}
Note that if $T$ has bounded dp-rank, then it is strongly dependent. 
\begin{rem}
All dp-minimal theories are of bounded dp-rank. This includes all
$o$-minimal theories and the $p$-adics. 
\end{rem}
The name is justified by the following fact:
\begin{fact}
\label{fac:dp-rank}\cite{AlexAlfKaplan} If $T$ has bounded dp-rank,
then for any $m<\omega$, there is some $n_{m}<\omega$ such that
a configuration as in Definition \ref{def:dp-rank} with $n_{m}$
replacing $n$ is impossible for a tuple $x$ of length $m$ (in fact
$n_{m}\leq m\cdot n_{1}$).\end{fact}
\begin{lem}
\label{lem:bounded dp rank case}Let $G$ be type definable group
in a bounded dp-rank theory $T$. 

Given a family of type definable subgroups $\left\{ p_{i}\left(x,a_{i}\right)\left|\, i<\omega\right.\right\} $
such that $\left\langle a_{i}\left|\, i<\omega\right.\right\rangle $
is an indiscernible sequence and $p_{2i}=p_{2i+1}$ for all $i<\omega$,
there is some $n<\omega$ and $i<n$ such that $\bigcap_{j\neq i,j<n}p_{j}\left(\C,a_{j}\right)=\bigcap_{j<n}p_{j}\left(\C,a_{j}\right)$.

In particular, if $p_{i}$ is constant (say $p$) and $\left\langle a_{i}\left|\, i<\omega\right.\right\rangle $
is an \uline{indiscernible set}, then $\bigcap_{i<\omega}p\left(\C,a_{i}\right)=\bigcap_{i<n}p\left(\C,a_{i}\right)$.

In particular, $T$ has Property A. \end{lem}
\begin{proof}
The proof is exactly the same as the proof of Lemma \ref{lem:strong1BalSaxl},
but we only need to construct $g_{n}$ for $n$ large enough. 
\end{proof}
Another similar proposition:
\begin{prop}
\label{prop:NormalIntersection} Assume $T$ is strongly dependent,
$G$ a type definable group and $G_{i}\leq G$ are type definable
\uline{normal} subgroups for $i<\omega$. Then there is some $i_{0}$
such that $\left[\bigcap_{i\neq i_{0}}G_{i}:\bigcap_{i<\omega}G_{i}\right]<\infty$.\end{prop}
\begin{proof}
Assume not. Then, for each $i<\omega$, we have an indiscernible sequence
$\left\langle a_{i,j}\left|\, j<\omega\right.\right\rangle $ (over
the parameters defining all the groups) such that $a_{i,j}\in\bigcap_{k\neq i}G_{k}$
and for $j_{1}<j_{2}<\omega$, $a_{i,j_{1}}^{-1}\cdot a_{i,j_{2}}\notin G_{i}$.
Note that if $d_{1},d_{2},d_{3}\in G_{i}$, then $d_{1}\cdot a_{i,j_{1}}^{-1}\cdot d_{2}\cdot a_{i,j_{2}}\cdot d_{3}\notin G_{i}$,
since $G_{i}$ is normal. By compactness there is a formula $\psi_{i}\left(x\right)$
in the type defining $G_{i}$ such that for all $d_{1},d_{2},d_{3}\in G_{i}$,
$\neg\psi_{i}\left(d_{1}\cdot a_{i,j_{1}}^{-1}\cdot d_{2}\cdot a_{i,j_{2}}\cdot d_{3}\right)$
holds (by indiscernibility it is the same for all $j_{1}<j_{2}$).
We may assume, applying Ramsey, that the array $\left\langle a_{i,j}\left|\, i,j<\omega\right.\right\rangle $
is indiscernible (i.e. the sequences are mutually indiscernible).
Let $\varphi_{i}\left(x,y\right)=\psi_{i}\left(x^{-1}\cdot y\right)$. 

Now we check that the set $\left\{ \varphi_{i}\left(x,a_{i,0}\right)\land\neg\varphi_{i}\left(x,a_{i,1}\right)\left|\, i<n\right.\right\} $
is consistent for each $n<\omega$. Let $c=a_{0,0}\cdot\ldots\cdot a_{n-1,0}$
(the order does not really matter, but for the proof it is easier
to fix one). So $\varphi_{i}\left(c,a_{i,0}\right)$ holds iff $\psi_{i}\left(a_{n-1,0}^{-1}\cdot\ldots\cdot a_{i,0}^{-1}\cdot\ldots\cdot a_{0,0}^{-1}\cdot a_{i,0}\right)$
holds. But since $G_{i}$ is normal, $a_{i,0}^{-1}\cdot\ldots\cdot a_{0,0}^{-1}\cdot a_{i,0}\in G_{i}$,
so the entire product is in $G_{i}$, so $\varphi_{i}\left(c,a_{i,0}\right)$
holds. On the other hand, $\psi_{i}\left(a_{n-1,0}^{-1}\cdot\ldots\cdot a_{i,0}^{-1}\cdot\ldots\cdot a_{0,0}^{-1}\cdot a_{i,1}\right)$
does not hold by choice of $\psi_{i}$.
\end{proof}
The following Corollary is a weaker version of Corollary \ref{cor:finiteIndex}:
\begin{cor}
If $G$ is an abelian definable group in a strongly dependent theory
and $S\subseteq\omega$ is an infinite set of pairwise co-prime numbers,
then for almost all (i.e. for all but finitely many) $n\in S$, $\left[G:G^{n}\right]<\infty$.
In particular, if $K$ is a definable field in a strongly dependent
theory, then for almost all primes $p$, $\left[K^{\times}:\left(K^{\times}\right)^{p}\right]<\infty$.\end{cor}
\begin{proof}
Let $K\subseteq S$ be the set of $n\in S$ such that $\left[G:G^{n}\right]<\infty$.
If $S\backslash K$ is infinite, we replace $S$ with $S\backslash K$.

For $i\in S$, let $G_{i}=G^{i}$ (so it is definable). By Proposition
\ref{prop:NormalIntersection}, there is some $n$ such that $\left[\bigcap_{i\neq n}G_{i}:\bigcap_{i\in S}G_{i}\right]<\infty$.
If $\left[G:G_{n}\right]=\infty$, then there is an indiscernible
sequence $\left\langle a_{i}\left|\, i<\omega\right.\right\rangle $
of elements of $G$, such that $a_{i}^{-1}\cdot a_{j}\notin G_{n}$.
Suppose $S_{0}\subseteq S\backslash\left\{ n\right\} $ is a finite
subset and let $r=\prod S_{0}$. Then $\left\langle a_{i}^{r}\left|\, i<\omega\right.\right\rangle $
is an indiscernible sequence in $G^{r}\subseteq\bigcap_{i\in S_{0}}G_{i}$
such that $a_{i}^{-r}\cdot a_{j}^{r}\notin G_{n}$. So by compactness,
we can find such a sequence in $\bigcap_{i\neq n}G_{i}$ --- contradiction.\end{proof}
\begin{rem}
The above Proposition and Corollary can be generalized (with almost
the same proofs) to the case where the theory is only \emph{strong}.
For the definition, see \cite{AdlerStrong}. 
\end{rem}

\begin{rem}
This Corollary generalizes in some sense \cite[Proposition 2.1]{PilKrup}
(as they only assumed finite weight of the generic type). And so,
as in \cite[Corollary 2.2]{PilKrup} , we can conclude that if $K$
is a field definable in a strongly stable theory (i.e. the theory
is strongly dependent and stable), then $K^{p}=K$ for almost all
primes $p$.\end{rem}
\begin{problem}
Is Proposition \ref{prop:NormalIntersection} is still true without
the assumption that the groups are normal? 

Note that in strongly dependent$^{2}$ theories, this assumption is
not needed: Let $H_{i}=\bigcap_{j<i}G_{i}$. Then $\left[H_{i}:H_{i+1}\right]<\infty$
for all $i$ big enough by Proposition \ref{prop:noInfiniteChain}.
But this implies $\left[\bigcap_{j\neq i}G_{j}:\bigcap_{j}G_{j}\right]<\infty$. 
\end{problem}

\subsection*{$\kappa$-intersection.\protect \\
}

This part is joint work with Frank Wagner.
\begin{defn}
For a cardinal $\kappa$ and a family $\mathfrak{F}$ of subgroups
of a group $G$, the $\kappa$ intersection $\bigcap_{\kappa}\mathfrak{F}$
is $\left\{ g\in G\left|\,\left|\left\{ F\in\mathfrak{F}\left|g\notin F\right.\right\} \right|<\kappa\right.\right\} $.\end{defn}
\begin{prop}
\label{prop:kappIntersection}Let $G$ be a type definable group in
a dependent theory. Suppose
\begin{itemize}
\item $\mathfrak{F}$ is a family of uniformly type definable subgroups
defined by $p\left(x,y\right)$.
\end{itemize}
Then for any regular cardinal $\kappa>\left|p\right|$ (in the sense
of Notation \ref{not:Size of a type}), and any subfamily $\mathfrak{G}\subseteq\mathfrak{F}$,
there is some $\mathfrak{G}'\subseteq\mathfrak{G}$ such that
\begin{itemize}
\item [$\star$]$\left|\mathfrak{G}'\right|<\kappa$ and $\bigcap\mathfrak{G}$
is $\bigcap\mathfrak{G}'\cap\bigcap_{\kappa}\mathfrak{G}$.
\end{itemize}
\end{prop}
\begin{proof}
Let $\kappa$ be such a cardinal. Assume that there is some family
$\mathfrak{G}=\left\{ H_{i}\left|\, i<\varkappa\right.\right\} $,
which is a counterexample of the proposition. For $g\in G$, let $J_{g}=\left\{ i<\varkappa\left|\, g\in H_{i}\right.\right\} $.
So $g\in\bigcap_{\kappa}\mathfrak{G}$ iff $\left|\varkappa\backslash J_{g}\right|$<$\kappa$.

For $i<\kappa$ we define by induction $g_{i}\in\bigcap_{\kappa}\mathfrak{G}$,
$I_{i}\subseteq\varkappa$, $R_{i}\subseteq\varkappa$ and $\alpha_{i}<\varkappa$
such that 
\begin{enumerate}
\item $R_{0}=\left[0,\alpha_{0}\right)$ and for $0<i$, $R_{i}=\bigcup_{j<i}R_{j}\cup\left[\left[\sup_{j<i}\alpha_{j},\alpha_{i}\right)\cap\bigcap_{j<i}I_{j}\right]$
(so $R_{i}\subseteq\alpha_{i}$)
\item $\bigcap_{j\leq i}J_{g_{j}}\subseteq R_{i}\cup I_{i}$ (so by the
definition of $\bigcap_{\kappa}$, and by the regularity of $\kappa$,$\left|\varkappa\backslash\left(R_{i}\cup I_{i}\right)\right|<\kappa$) 
\item $\bigcap_{\kappa}\mathfrak{G}\cap\bigcap_{j<i}H_{\alpha_{j}}\subseteq\bigcap_{\alpha\in R_{i}}H_{\alpha}$
\item $I_{i}\cap\left[0,\alpha_{i}\right]=\emptyset$
\item $I_{i}$ is $\subseteq$-decreasing
\item $\alpha_{i}$ is $<$-increasing
\item $I_{i}\subseteq J_{g_{i}}$
\item For $j<i$, $g_{i}\in H_{\alpha_{j}}$, $g_{j}\in H_{\alpha_{i}}$
and $g_{i}\notin H_{\alpha_{i}}$
\end{enumerate}
Let $\alpha_{0}<\varkappa$ be minimal such that there is some $g_{0}\in\bigcap_{\kappa}\mathfrak{G}\backslash H_{\alpha_{0}}$
(it must exist, otherwise $\bigcap_{\kappa}\mathfrak{G}=\bigcap\mathfrak{G}$).
Let $I_{0}=\left\{ j>\alpha_{0}\left|\, g_{\alpha_{0}}\in H_{j}\right.\right\} $. 

For $\alpha_{0}$, (2), (3) and (4) are true, by the definition of
$\bigcap_{\kappa}$ and the choice of $\alpha_{0}$. 

Suppose we have chosen $g_{j}$, $I_{j}$ and $\alpha_{j}$ (so $R_{j}$
is already defined by (1)) for $j<i$. 

Let $J=\bigcap_{j<i}I_{j}$. Choose $g_{i}\in\left(\bigcap_{\kappa}\mathfrak{G}\cap\bigcap_{j<i}H_{\alpha_{j}}\right)\backslash H_{\alpha_{i}}$
where $\alpha_{i}\in J$ is the smallest possible such that this set
is nonempty. Suppose for contradiction that we cannot find such $\alpha_{i}$,
then $\bigcap_{\kappa}\mathfrak{G}\cap\bigcap_{j<i}H_{\alpha_{j}}\subseteq\bigcap_{\alpha\in J}H_{\alpha}$,
so 
\[
\bigcap_{\kappa}\mathfrak{G}\cap\bigcap_{j<i}H_{\alpha_{j}}\cap\bigcap_{j\in\varkappa\backslash J}H_{j}=\bigcap\mathfrak{G}.
\]
Let $J'=J\cup\bigcup_{j<i}R_{j}$, then by (3), $\bigcap\mathfrak{G}$
equals 
\[
\bigcap_{\kappa}\mathfrak{G}\cap\bigcap_{j<i}H_{\alpha_{j}}\cap\bigcap_{j\in\varkappa\backslash J'}H_{j}.
\]
Note that $\bigcap_{j<i}\left(R_{j}\cup I_{j}\right)\subseteq J'$,
so by regularity of $\kappa$, and by (2), $\left|\varkappa\backslash J'\right|<\kappa$,
so we get a contradiction. 

Let $I_{i}=\left\{ \alpha_{i}<j\in J\left|\, g_{i}\in H_{j}\right.\right\} $,
and let us check the conditions above.

Conditions (4) -- (7) are easy.

Condition (2): By induction we have
\[
\bigcap_{j\leq i}J_{g_{j}}=\bigcap_{j<i}J_{g_{j}}\cap J_{g_{i}}\subseteq J'\cap J_{g_{i}}\subseteq R_{i}\cup\left(J\cap J_{g_{i}}\right)
\]
But by (4) and the definition of $R_{i}$, letting $\alpha=\sup_{j<i}\alpha_{j}$,
we have 
\[
J\cap J_{g_{i}}\subseteq\left[\left[\alpha,\alpha_{i}\right)\cap\bigcap_{j<i}I_{j}\right]\cup I_{i}\subseteq R_{i}\cup I_{i}
\]
Condition (3) is true by the minimality of $\alpha_{i}$: $\bigcap_{\kappa}\mathfrak{G}\cap\bigcap_{j<i}H_{\alpha_{j}}\subseteq\bigcap_{\beta\in J\cap\left[\alpha,\alpha_{i}\right)}H_{\beta}$,
so by the induction hypothesis, we are done. 

Condition (8): We show that $g_{j}\in H_{\alpha_{i}}$ for $j<i$.
We have that $\alpha_{i}\in J$ so also in $I_{j}$ which, by (7)
is a subset of $J_{g_{j}}$, so $g_{j}\in H_{\alpha_{i}}$. 

Finally, we have that for each $i,j<\kappa$, $g_{i}\in H_{\alpha_{j}}$
iff $i\neq j$. But by Lemma \ref{lem:BaldSaxType}, there is some
$i_{0}<\left|p\right|^{+}$ such that $\bigcap_{i\neq i_{0}}H_{\alpha_{i}}=\bigcap_{i<\left|p\right|^{+}}H_{\alpha_{i}}$
--- contradiction.
\end{proof}

\section{\label{sec:An-example}A counterexample}

In this section we shall present an example that shows that Property
A does not hold in general dependent (or even stable) theories.

Let $S=\left\{ u\subseteq\omega\left|\,\left|u\right|<\omega\right.\right\} $,
and $V=\left\{ f:S\to2\left|\,\left|\supp\left(f\right)\right|<\infty\right.\right\} $
where $\supp\left(f\right)=\left\{ x\in S\left|\, f\left(x\right)\neq0\right.\right\} $.
This has a natural group structure as a vector space over $\Ff_{2}=\mathbb{Z}/2\mathbb{Z}$.

For $n,m<\omega$, define the following groups:
\begin{itemize}
\item $G_{n}=\left\{ f\in V\left|\, u\in\supp\left(f\right)\Rightarrow\left|u\right|=n\right.\right\} $
\item $G_{\omega}=\prod_{n}G_{n}$
\item $G_{n,m}=\left\{ f\in V\left|\, u\in\supp\left(f\right)\Rightarrow\left|u\right|=n\,\&\, m\in u\right.\right\} $
(so $G_{0,m}=0$)
\item $H_{n,m}=\left\{ \eta\in G_{\omega}\left|\,\eta\left(n\right)\in G_{n,m}\right.\right\} $
\end{itemize}
Now we construct the model:

Let $L$ be the language (vocabulary) $\left\{ P,Q\right\} \cup\left\{ R_{n}\left|\, n<\omega\right.\right\} \cup L_{AG}$
where $L_{AG}$ is the language of abelian groups, $\left\{ 0,+\right\} $;
$P$ and $Q$ are unary predicates; and $R_{n}$ is binary. Let $M$
be the following $L$-structure: $P^{M}=G_{\omega}$ (with the group
structure), $Q^{M}=\omega$ and $R_{n}=\left\{ \left(\eta,m\right)\left|\,\eta\in H_{n,m}\right.\right\} $.
Let $T=Th\left(M\right)$.

Let $p\left(x,y\right)$ be the type $\bigcup\left\{ R_{n}\left(x,y\right)\left|\, n<\omega\right.\right\} $.
Note that since $H_{n,m}$ is a subgroup of $G_{\omega}$, for each
$m<\omega$, $p\left(M,m\right)$ is a subgroup of $G_{\omega}$.
\begin{claim}
Let $N\models T$ be $\aleph_{1}$-saturated. For any $m$, and any
distinct $\alpha_{0},\ldots,\alpha_{m}\in P^{N}$, $\bigcap_{i\leq m}p\left(N,\alpha_{i}\right)$
is different than any sub-intersection of size $m$. \end{claim}
\begin{proof}
We show that $\bigcap_{i\leq m}p\left(N,\alpha_{i}\right)\subsetneq\bigcap_{i<m}p\left(N,\alpha_{i}\right)$
(the general case is similar). More specifically, we show that 
\[
\bigcap_{i<m}p\left(N,\alpha_{i}\right)\backslash\bigcap_{i\leq m}R_{m}\left(N,\alpha_{i}\right)\neq\emptyset.
\]
By saturation, it is enough to show that this is the case in $M$,
so we assume $M=N$. Note that if $\eta\in\bigcap_{i\leq m}R_{m}\left(M,\alpha_{i}\right)$,
then $\eta\in H_{m,\alpha_{i}}$ for all $i\leq m$. So for all $i\leq m$,
$u\in\supp\left(\eta\left(n\right)\right)\Rightarrow\left|u\right|=m\,\&\,\alpha_{i}\in u$.
This implies that $\supp\left(\eta\left(m\right)\right)=\emptyset$,
i.e. $\eta\left(m\right)=0$. But we can find $\eta\in\bigcap_{i<m}p\left(M,\alpha_{i}\right)$
such that $\eta\left(m\right)\neq0$, for instance let $\eta\left(n\right)=0$
for all $n\neq m$ while $\left|\supp\left(\eta\left(m\right)\right)\right|=1$
and $\eta\left(m\right)\left(\left\{ \alpha_{0},\ldots,\alpha_{m-1}\right\} \right)=1$. 
\end{proof}
Next we shall show that $T$ is stable. For this we will use $\kappa$
resplendent models. This is a very useful (though not a very well
known) tool for proving that theories are stable, and we take the
opportunity to promote it.
\begin{defn}
Let $\kappa$ be a cardinal. A model $M$ is called $\kappa$-resplendent
if whenever
\begin{itemize}
\item $M\prec N$; $N'$ is an expansion of $N$ by less than $\kappa$
many symbols; $\bar{c}$ is a tuple of elements from $M$ and $\lg\left(\bar{c}\right)<\kappa$
\end{itemize}
There exists an expansion $M'$ of $M$ to the language of $N'$ such
that $\left\langle M',\bar{c}\right\rangle \equiv\left\langle N',\bar{c}\right\rangle $.

The following remarks are not crucial for the rest of the proof.\end{defn}
\begin{rem}
\cite{Sh:363}\end{rem}
\begin{enumerate}
\item If $\kappa$ is regular and $\kappa>\left|T\right|$, and $\lambda=\lambda^{<\kappa}$,
then $T$ has a $\kappa$-resplendent model of size $\lambda$. 
\item A $\kappa$ resplendent model is also $\kappa$-saturated.
\item If $M$ is $\kappa$ resplendent then $M^{\eq}$ is also such.
\end{enumerate}
The following is a useful observation:
\begin{claim}
\label{cla:SameSizeResplendent}If $M$ is $\kappa$-resplendent for
some $\kappa$, and $A\subseteq M$ is definable and infinite, then
$\left|A\right|=\left|M\right|$.\end{claim}
\begin{proof}
Enrich the language with a function symbol $f$. Let $T'=T\cup\left\{ f:M\to A\mbox{ is injective}\right\} $.
Then $T'$ is consistent with an elementary extension of $M$ (for
example, take an extension $N$ of $M$ where $\left|A\right|=\left|M\right|$,
and then take an elementary substructure $N'\prec N$ of size $\left|M\right|$
containing $M$ and $A^{N}$). Hence we can expand $M$ to a model
of $T'$. 
\end{proof}
The main fact is
\begin{thm}
\cite[Main Lemma 1.9]{Sh:363}Assume $\kappa$ is regular and $\lambda=\lambda^{\kappa}+2^{\left|T\right|}$.
Then, if $T$ is unstable then $T$ has $>\lambda$ pairwise nonisomorphic
$\kappa$-resplendent models of size $\lambda$%
\footnote{In fact, by \cite[Claim 3.1]{Sh:363}, if $T$ is unstable there are
$2^{\lambda}$ such models.%
}. On the other hand, if $T$ is stable and $\kappa\geq\kappa\left(T\right)+\aleph_{1}$
then every $\kappa$-resplendent model is saturated.\end{thm}
\begin{prop}
$T$ is stable.\end{prop}
\begin{proof}
We may restrict $T$ to a finite sub-language, $L_{n}=\left\{ P,Q,\right\} \cup\left\{ R_{i}\left|\, i<n\right.\right\} \cup L_{AG}$
. 

Our strategy is to prove that our theory has a unique model in size
$\lambda$ which is $\kappa$ resplendent where $\kappa=\aleph_{0}$,
$\lambda=2^{\aleph_{0}}$. Let $N_{0},N_{1}$ be two $\kappa$-resplendent
models of size $\lambda$. 

By Claim \ref{cla:SameSizeResplendent}, $\left|Q^{N_{0}}\right|=\left|Q^{N_{1}}\right|=\lambda$
and we may assume that $Q^{N_{0}}=Q^{N_{1}}=\lambda$. 

Let $G_{0}=P^{N_{0}}$ and $G_{1}=P^{N_{1}}$ with the group structure.
For $i<n$, $j<2$ and $\alpha<\lambda$, let $H_{i,\alpha}^{j}=\left\{ x\in G_{j}\left|\, R_{i}^{N_{j}}\left(x,\alpha\right)\right.\right\} $.
This is a definable subgroup of $G_{j}$. For $k\leq n$, let $G_{j}^{k}=\bigcap_{\alpha<\lambda,\, i\neq k,\, i<n}H_{i,\alpha}^{j}$.
In our original model $M$, this group is $\left\{ \eta\in G_{\omega}\left|\,\forall i\neq k,\, i<n\left(\eta\left(i\right)=0\right)\right.\right\} $.
Note that $G_{j}=\sum_{k<n}G_{j}^{k}$, and that $G_{j}^{k_{0}}\cap\sum_{k<n,k\neq k_{0}}G_{j}^{k}=G_{j}^{n}$
(this is true in our original model $M$, so it is part of the theory).
We give each $G_{j}^{k}$ the induced $L$-structure $N_{j}^{k}=\left\langle G_{j}^{k},\lambda\right\rangle $,
i.e. we interpret $R_{i}^{N_{j}^{k}}=R_{i}\cap\left(G_{k}^{j}\times\lambda\right)$. 

Since these groups are definable and infinite, their cardinality is
$\lambda$, and hence their dimension (over $\Ff_{2}$) is $\lambda$.
In particular there is a group isomorphism $f_{n}:G_{0}^{n}\to G_{1}^{n}$.
Note that $f_{n}$ is an isomorphism of the induced structure on $N_{j}^{n}=\left\langle G_{j}^{n},\lambda\right\rangle $.
\begin{claim*}
For $k<n$, there is an isomorphism $f_{k}:G_{0}^{k}\to G_{1}^{k}$
which is an isomorphism of the induced structure $N_{j}^{k}=\left\langle G_{j}^{k},\lambda\right\rangle $
and extends $f_{n}$.
\end{claim*}
Assuming this claim, we shall finish the proof. Define $f:G_{0}\to G_{1}$
by: given $x\in G_{0}$, write it as a sum $\sum_{k<n}x_{k}$ where
$x_{k}\in G_{0}^{k}$, and define $f\left(x\right)=\sum_{k<n}f\left(x_{k}\right)$.
This is well defined because if $\sum_{k<n}x_{k}=\sum_{k<n}x'_{k}$
then $\sum_{k<n}\left(x_{k}-x_{k}'\right)=0$ so for all $k<n$, $x_{k}-x_{k}'\in G_{0}^{n}$,
so
\begin{eqnarray*}
\sum_{k<n}\left(f\left(x_{k}\right)-f\left(x_{k}'\right)\right) & = & \sum_{k<n}\left(f\left(x_{k}-x_{k}'\right)\right)=\sum_{k<n}\left(f_{n}\left(x_{k}-x_{k}'\right)\right)=\\
 & = & f_{n}\left(\sum_{k<n}x_{k}-x_{k}'\right)=f_{n}\left(0\right)=0.
\end{eqnarray*}
It is easy to check similarly that $f$ is a group isomorphism. Also,
$f$ is an $L_{n}$-isomorphism because if $R_{i}^{N_{0}}\left(a,\alpha\right)$
for some $i<n$, $\alpha<\lambda$ and $a\in G_{0}$, then write $a=\sum_{k<n}a_{k}$
where $a_{k}\in G_{0}^{k}$. Since $R_{i}^{N_{0}}\left(a,\alpha\right)$
and $R_{i}^{N_{0}}\left(a_{k},\alpha\right)$ for all $k\neq i$,
it follows that $R_{i}^{N_{0}}\left(a_{i},\alpha\right)$ holds, so
$R_{i}^{N_{1}}\left(f_{k}\left(a_{k}\right),\alpha\right)$ holds
for all $k<n$, and so $R_{i}^{N_{1}}\left(f\left(a\right),\alpha\right)$
holds. The other direction is similar. 
\begin{proof}
\renewcommand{\qedsymbol}{} (of claim) For a finite set $b$ of elements
of $\lambda$, let $L_{b}^{j}=G_{j}^{k}\cap\bigcap_{\alpha\in b}H_{k,\alpha}^{j}$.
For $m\leq k+1$, let $K_{m}^{j}=\sum_{\left|b\right|=m}L_{b}^{j}$
(as a subspace of $G_{k}^{j}$), so $K_{m}^{j}$ is not necessarily
definable (however $K_{0}^{j}$ and $K_{k+1}^{j}$ are). So this is
a decreasing sequence of subgroups (so subspaces), $G_{j}^{k}=K_{0}^{j}\geq\ldots\geq K_{k+1}^{j}=G_{j}^{n}$.
Now it is enough to show that
\begin{subclaim*}
For $m\leq k+1$, there is an isomorphism $f_{m}:K_{m}^{0}\to K_{m}^{1}$
which is an isomorphism of the induced structure $\left\langle K_{m}^{j},\lambda\right\rangle $. \end{subclaim*}
\begin{proof}
(of subclaim) The proof is by reverse induction. For $m=k+1$ we already
have this. Suppose we have $f_{m+1}$ and we want to construct $f_{m}$.
Let $b\subseteq\lambda$ of size $m$. If $m=k$, then it is easy
to see that $\left|L_{b}^{j}/\left(K_{m+1}^{j}\cap L_{b}^{j}\right)\right|=2$
(this is true in $M$), so there is an isomorphism $g_{b}:L_{b}^{0}/\left(K_{m+1}^{0}\cap L_{b}^{0}\right)\to L_{b}^{1}/\left(K_{m+1}^{1}\cap L_{b}^{1}\right)$.

Assume $\left|b\right|<k$. In our original model $M$, $L_{b}\subseteq K_{k}$,
but here can find infinitely pairwise distinct cosets in $L_{b}^{j}/\left(K_{m+1}^{j}\cap L_{b}^{j}\right)$.
Indeed, we can write a type in $\lambda$ infinitely many variables
$\left\{ x_{i}\left|\, i<\lambda\right.\right\} $ over $b$ saying
that $x_{i}\in L_{b}$ and $x_{i}-x_{j}\notin K_{m+1}$ for $i\neq j$
--- for all $r<\omega$, it will contain a formula of the form
\[
\forall\left(z_{0},\ldots,z_{r-1}\right)\forall_{t<r}\left(\bar{y}_{t}\right)\left(\left[\forall t<r\left(z_{t}\in L_{\bar{y}_{t}}\land\left|\bar{y}_{t}\right|=m+1\right)\right]\to x_{i}-x_{j}\neq\sum_{t=0}^{r-1}z_{t}\right).
\]
To show that this type is consistent, we may assume that $b\subseteq Q^{M}$
so we work in our original model $M$. For such $r$ and $b$, choose
distinct $\eta_{0},\ldots\eta_{l-1}\in G_{\omega}$ such that for
$s,s'<l$
\begin{itemize}
\item $\eta_{s}\left(i\right)=0$ for $i\neq k$
\item $\left|\supp\left(\eta_{s}\left(k\right)\right)\right|=r+1$
\item $u_{1}\in\supp\left(\eta_{s}\left(k\right)\right)\,\&\, u_{2}\in\supp\left(\eta_{s'}\left(k\right)\right)\Rightarrow u_{1}\cap u_{2}=b$
($s$ might be equal to $s'$)
\end{itemize}
Then $\left\{ \eta_{s}\left|\, s<l\right.\right\} $ is such that
$\eta_{s_{1}},\eta_{s_{2}}$ satisfies the formula above for all $s_{1}\neq s_{2}<l$
(assume $z_{0}\in L_{c_{0}},\ldots,z_{r-1}\in L_{c_{r}}$ where $\left|c_{t}\right|=m+1$
and $\sum_{t<r}z_{t}=\eta_{s_{1}}-\eta_{s_{2}}$. We may assume that
\[
\bigcup_{t<r}\supp\left(z_{t}\right)=\supp\left(\eta_{s_{1}}-\eta_{s_{2}}\right)=\supp\left(\eta_{s_{1}}\right)\cup\supp\left(\eta_{s_{2}}\right),
\]
but then for $t<r$, $\left|\supp\left(z_{t}\right)\right|\leq1$
by our choice of $\eta_{s}$ and this is a contradiction).

Now, let $N_{j}'$ be an elementary extension of $N_{j}$ with realizations
$D=\left\{ c_{i}\left|\, i<\lambda\right.\right\} $ of this type,
and we may assume $\left|N_{j}'\right|=\lambda$. Then, add a predicate
for the set $D$, and an injective function from $N_{j}'$ to $D$.
Finally, by resplendence of $N_{j}$, $\left|L_{b}^{j}/\left(K_{m+1}^{j}\cap L_{b}^{j}\right)\right|=\lambda$. 

Hence it has a basis of size $\lambda$, and let $g_{b}:L_{b}^{0}/\left(K_{m+1}^{0}\cap L_{b}^{0}\right)\to L_{b}^{1}/\left(K_{m+1}^{1}\cap L_{b}^{1}\right)$
be an isomorphism of $\Ff_{2}$-vector spaces. 

Note that $f_{m+1}\upharpoonright K_{m+1}^{0}\cap L_{b}^{0}$ is onto
$K_{m+1}^{1}\cap L_{b}^{1}$ (this is because $f_{m+1}$ is an isomorphism
of the induced structure). We can write $L_{b}^{j}=\left(K_{m+1}^{j}\cap L_{b}^{j}\right)\oplus W^{j}$
where $W^{j}\cong L_{b}^{j}/\left(K_{m+1}^{j}\cap L_{b}^{j}\right)$,
so $g_{b}$ induces an isomorphism from $W^{0}$ to $W^{1}$. Now
extend $f_{m+1}\upharpoonright K_{m+1}^{0}\cap L_{b}^{0}$ to $f_{m}^{b}:L_{b}^{0}\to L_{b}^{1}$
using $g_{b}$. 

Next, note that $\left\{ L_{b}^{j}\left|\, b\subseteq\lambda,\,\left|b\right|=m\right.\right\} $
is independent over $K_{m+1}^{j}$, i.e. for distinct $b_{0},\ldots,b_{r}$,
$L_{b_{r}}^{j}\cap\sum_{t<r}L_{b_{r}}^{j}\subseteq K_{m+1}^{j}$.
Indeed, in our original model $M$, the intersection $L_{b_{r}}\cap\sum_{t<r}L_{b_{t}}$
is equal to $\sum_{t<r}L_{b_{r}\cup b_{t}}$, so this is true also
in $N_{j}$ (in fact, this is true for every choice of finite sets
$b_{t}$ --- regardless of their size). 

Define $f_{m}$ as follows: given $a\in K_{m}^{j}$, we can write
$a=\sum_{b\in B}a_{b}$ where $a_{b}\in L_{b}$ for a finite $B\subseteq\left\{ b\subseteq\lambda\left|\,\left|b\right|=m\right.\right\} $,
and define $f_{m}\left(a\right)=\sum f_{b}\left(a_{b}\right)$. It
is well defined: if $\sum_{b\in B}x_{b}=\sum_{b'\in B'}y_{b'}$, then
for $b_{1}\in B\cap B'$, $b_{2}\in B\backslash B'$ and $b_{3}\in B'\backslash B$,
$\left(x_{b_{1}}-y_{b_{1}}\right),x_{b_{2}},y_{b_{3}}\in K_{m+1}$,
so 
\begin{eqnarray*}
 & \sum_{b\in B}f_{b}\left(x_{b}\right)-\sum_{b'\in B'}f_{b'}\left(y_{b'}\right)=\\
 & \sum_{b\in B\cap B'}f_{m+1}\left(x_{b}-y_{b}\right)+\sum_{b\in B\backslash B'}f_{m+1}\left(x_{b}\right)-\sum_{b\in B'\backslash B}f_{m+1}\left(y_{b}\right) & =0.
\end{eqnarray*}
It is easy to check similarly that $f_{m}$ is a group isomorphism.

We check that $f_{m}$ is an isomorphism of the induced structure.
So suppose $a\in K_{m}^{0}$ , $\alpha<\lambda$ and $i<\omega$.
If $i\neq k$, then since $K_{m}^{j}\subseteq G_{j}^{k}$ for $j<2$,
both $R_{i}^{N_{0}}\left(a,\alpha\right)$ and $R_{i}^{N_{1}}\left(f\left(a\right),\alpha\right)$
hold. Suppose $R_{k}^{N_{0}}\left(a,\alpha\right)$ holds. Write $a=\sum_{b\in B}a_{b}$
as above. Then (by the remark in parenthesis above) we may assume
that $b\in B\Rightarrow\alpha\in b$. So by definition of $f_{m}$,
$R_{k}^{N_{1}}\left(f_{m}\left(a_{\alpha}\right),\alpha\right)$ holds.
The other direction holds similarly and we are done.
\end{proof}
\end{proof}
\end{proof}
\begin{note}
This example is not strongly dependent, because the sequence of formulas
$R_{n}\left(x,y\right)$ is a witness of that the theory is not strongly
dependent. So as we said in the introduction, it is still open whether
Property A holds for strongly dependent theories.
\end{note}
\bibliographystyle{alpha}
\bibliography{common}

\end{document}